\documentclass[preprint,12pt]{elsarticleHF}

\usepackage{amssymb}

\usepackage{amsmath, amsthm}
\newtheorem{theorem}{Theorem}

\newproof{pf}{Proof}
\newtheorem{remark}{Remark}
\newtheorem{corollary}{Corollary}
\newtheorem{proposition}{Proposition}
\newtheorem{example}{Example}

\usepackage{color}
\usepackage{csquotes}

\usepackage{graphicx}
\usepackage{xcolor}
\usepackage{tikz}
\usepackage{latexsym, amsmath}
\usetikzlibrary{matrix}
\usetikzlibrary{intersections}
\usetikzlibrary{calc}
\usetikzlibrary{positioning}
\usetikzlibrary{arrows}
\usetikzlibrary{fit}
\usepackage{textpos}

\journal{LAA}

\begin{document}

\begin{frontmatter}



\title{On vector spaces of linearizations for matrix polynomials in orthogonal bases}


\author[label1]{Heike Fa\ss bender}
\address[label1]{Institut \emph{Computational Mathematics}/ AG Numerik, TU Braunschweig, Pockelsstr. 14, 38106 Braunschweig, Germany}
\cortext[cor1]{Corresponding author, Email philip.saltenberger@tu-braunschweig.de}
\author[label1]{Philip Saltenberger\corref{cor1}}

\begin{abstract}
Regular and singular matrix
polynomials $P(\lambda) = \sum_{i=0}^k P_i\phi_i(\lambda), P_i \in \mathbb{R}^{n \times n}$ given in an
orthogonal basis $\phi_0(\lambda), \phi_1(\lambda), \ldots, \phi_k(\lambda)$ are considered. Following the ideas in \cite{MacMMM06}, the 
 vector spaces, called $\mathbb{M}_1(P)$, $\mathbb{M}_2(P)$
and $\mathbb{DM}(P)$, of potential linearizations for $P(\lambda)$ are
analyzed. All pencils in $\mathbb{M}_1(P)$ are characterized concisely.
Moreover, several easy to check criteria whether a pencil in $\mathbb{M}_1(P)$ is a (strong) linearization of
$P(\lambda)$ 
are given. The equivalence of some of them to the Z-rank-condition \cite{MacMMM06} is pointed out.
Results on the vector space dimensions, the genericity of linearizations in $\mathbb{M}_1(P)$ and the form of block-symmetric
pencils are derived in a new way on a basic algebraic level. Moreover, an extension of these results to degree-graded bases is presented. Throughout the paper, structural resemblances between
the matrix pencils in $\mathbb{L}_1$, i.e. the results obtained in \cite{MacMMM06}, and their generalized versions are pointed
out.
\end{abstract}


\end{frontmatter}



\section{Introduction}
Linearization of matrix polynomials expressed in standard and nonstan\-dard bases have received much attention in recent years.
In the ground-breaking paper \cite{MacMMM06} vector spaces 
of possible linearizations of matrix polynomials have been
introduced. These turned out to build an elegant framework to find and construct linearizations for square matrix polynomials as
well as to study their algebraical and analytical properties. While the paper \cite{MacMMM06} is mainly concerned with the characterization
and analysis of these spaces for matrix polynomials in the standard monomial basis, recently the research on matrix polynomials
and linearizations expressed in nonstandard polynomial bases has received more attention, see, e.g., \cite{AmiCL09,TerDM09,
FassSCheb16,EffK12,LawP16,NakNT12,NofP15,RobVVD16}.

This paper is devoted to the study of regular and singular matrix polynomials
$P(\lambda) = \sum_{i=0}^k P_i\phi_i(\lambda),$ $P_i \in
\mathbb{R}^{n \times n}$
 expressed in an orthogonal basis $\{ \phi_i (\lambda) \}_{i=0}^k$, generalizing most concepts from \cite{MacMMM06}
to this special case. In particular, we will consider the set $\mathbb{M}_1(P)$ of all $kn \times kn$ matrix pencils $\mathcal{L}(\lambda)$ satisfying
\[
{\mathcal L}(\lambda) (\Phi_k(\lambda)\otimes I_n) = v \otimes P(\lambda)
\]
with $\Phi_k(\lambda) := [\phi_{k-1} \cdots \phi_1 ~~\phi_0]^T.$
For the monomial basis,
this is just the definition of $\mathbb{L}_1(P)$ \cite[Definition 3.1]{MacMMM06} with
$\Phi_k(\lambda) = [ \, \lambda^{k-1} \; \cdots \; \lambda \; 1 \, ]^T =:\Lambda_k(\lambda).$
The same kind of generalization of $\mathbb{L}_1(P)$ to matrix polynomials in nonstandard bases has been already
considered, e.g.,  in \cite{TerDM09, NakNT12}. We will give an explicit characterization of the elements of
$\mathbb{M}_1(P)$ that enables us to formulate our results readily accessible providing quite short proofs.  Moreover, we
show how to easily construct linearizations by means of an intuitive and readily checked linearization condition.
Clearly, most of our findings are equivalent to already known results.
Thus our main contribution here is a new view
aiming to open up new perspectives on the structure of
ansatz spaces in general and present even well-known facts in a new livery.
A second main goal is to present the facts in a concise and succinct manner keeping the proofs on a basic algebraic level
without drawing on deeper theoretical
results. We present our results assuming the field underlying our derivations are
the real numbers $\mathbb{R}$. However, we expect that most of the concepts immediately extend (appropriately adjusted) to arbitrary fields $\mathbb{F}$, in particular to the complex numbers $\mathbb{C}$.

In Section \ref{sec2} the basic notation used and some well-known results are summarized. In Section \ref{sec3}, generalized
ansatz spaces for orthogonal bases are defined and their basic properties are proven. Section
\ref{sec:eigenvectorrecovery} is concerned with the eigenvector recovery, while in Section \ref{sec:singular} singular matrix polynomials are considered.
The extension of the double ansatz space
from \cite{MacMMM06} to orthogonal bases is the subject of Section \ref{sec4}, whereas Section \ref{sec5} provides a construction
algorithm for block-symmetric pencils. Section \ref{sec:eigenvectorexclusion} presents a partial generalization of
the eigenvalue exclusion theorem, while Section \ref{sec:polbases} is dedicated to the question how the results presented up to Section \ref{sec:eigenvectorexclusion} may be derived when an arbitrary degree-graded polynomial basis is considered instead of an orthogonal basis. In Section \ref{sec6} some concluding remarks are given.

\section{Preliminaries and Basic Notation}\label{sec2}
For $\mathbb{R}[\lambda]$, the ring of real polynomials in the variable $\lambda$, the $n \times n$ matrix
ring over $\mathbb{R}[\lambda]$ is denoted by $\mathbb{R}[\lambda]^{n \times n}$. Its elements are referred to as matrix
polynomials. Notice that $\mathbb{R}[\lambda]^{n \times n}$ is a vector space over $\mathbb{R}$.
We consider matrix polynomials $P(\lambda) \in \mathbb{R}[\lambda]^{n \times n}$ expressed in polynomial bases $\Phi = \lbrace
\phi_j(\lambda)
\rbrace_{j=0}^{\infty}$ that follow a three-term recurrence relation. In particular we assume that
\begin{equation}\label{def_orthbasis}
 \alpha_{j} \phi_{j+1}(\lambda) = (\lambda - \beta_j) \phi_j(\lambda) - \gamma_j \phi_{j-1}(\lambda) \qquad j \geq 0
\end{equation}
for some coefficients $\alpha_j \neq 0, \beta_j, \gamma_j \in \mathbb{R}$ and $\phi_{-1}(\lambda)=0, \phi_0(\lambda)=1$.
Popular special
cases include the monomials, Newton and Chebyshev bases or the Legendre basis. Moreover, we usually
assume that $P(\lambda) \in \mathbb{R}[\lambda]^{n \times n}$ may be expressed as
\begin{equation} P(\lambda) = P_k \phi_k(\lambda) + P_{k-1} \phi_{k-1}(\lambda) + \cdots + P_1 \phi_1(\lambda) + P_0
\phi_0(\lambda) \label{expr_P} \end{equation}
with $P_k \neq 0$. In this case $P(\lambda)$ is said to have degree $k$, i.e. $\textnormal{deg}(P(\lambda))=k$. A matrix
polynomial with $\textnormal{det}(P(\lambda)) \neq 0$ is called regular, otherwise it is called singular. Moreover, matrix
polynomials of degree one are called matrix pencils.

Suppose $P(\lambda) \in \mathbb{R}[\lambda]^{n \times n}$ is regular. Then any scalar $\alpha \in \mathbb{C}$ such that $P(\alpha) \in \mathbb{C}^{n \times n}$ is singular is called a finite eigenvalue of
$P(\lambda)$. The corresponding eigenspace is defined to be $\textnormal{null}(P(\alpha))$, i.e. the nullspace of
$P(\alpha)$. For any $n \times n$ matrix $A,$ $\mathcal{N}_r(A)$ denotes the right
nullspace of $A$, i.e. the set of all $x \in \mathbb{C}^n$ satisfying $Ax=0$, whereas $\mathcal{N}_\ell(A)$ is the set
of all $x \in \mathbb{C}^n$ that satisfy $x^TA=0$. If $P(\lambda)$ is a singular matrix polynomial, then the left and right nullspaces of $P(\lambda)$ (\cite[Def. 2.1]{TerDM09}) are defined as
\begin{align*}
\mathcal{N}_\ell(P) &= \big\lbrace x(\lambda) \in \mathbb{R}(\lambda)^n \; \big| \; x^T(\lambda)P(\lambda) = 0 \big\rbrace \\
\mathcal{N}_r(P) &= \big\lbrace y(\lambda) \in \mathbb{R}(\lambda)^n \; \big| \; P(\lambda)y(\lambda) = 0 \big\rbrace.
\end{align*}
Here $\mathbb{R}(\lambda)$ denotes the field of rational functions over $\mathbb{R}$.

Whenever $P(\lambda) \in \mathbb{R}[\lambda]^{n \times n}$ has degree $k$, the reversal of $P(\lambda)$ is the
matrix polynomial
$$ \textnormal{rev}_k(P(\lambda)) := \lambda^k P \left( \tfrac{1}{\lambda} \right)$$
of which it can be proven that its nonzero finite eigenvalues are the reciprocals of those of
$P(\lambda).$ Moreover, if zero is an eigenvalue of $\textnormal{rev}_k(P(\lambda))$, we say that $\infty$ is an eigenvalue of
$P(\lambda)$.

Assume $P(\lambda) \in \mathbb{R}[\lambda]^{n \times n}$ has degree $k$. Then a $kn \times kn$ matrix pencil
$\mathcal{L}(\lambda)=X \lambda + Y$ is called a linearization for $P(\lambda)$
if there exist two matrix polynomials $U(\lambda), V(\lambda) \in \mathbb{R}[\lambda]^{kn \times kn}$ with nonzero, real
determinants such that
 \[
U(\lambda) \mathcal{L}(\lambda) V(\lambda) = \left[ \begin{array}{c|c} P(\lambda) &
\begin{array}{ccc} 0_n & \cdots & 0_n \end{array} \\ \hline \begin{array}{c} 0_n \\ \vdots \\ 0_n \end{array} & I_{(k-1)n}
\end{array} \right]
\]
holds. Here $I_n$ denotes the $n \times n$ identity matrix, whereas $0_n$ is the $n \times n$ matrix of all zeros.
A linearization $\mathcal{L}(\lambda)$ for $P(\lambda)$ is called strong whenever $\textnormal{rev}_1(\mathcal{L}(\lambda))$ is a
linearization for $\textnormal{rev}_k(P(\lambda))$ as well. In case $\mathcal{L}(\lambda)$ is a strong linearization of a
matrix polynomial $P(\lambda)$, $\mathcal{L}(\lambda)$ and $P(\lambda)$ share the same finite and infinite eigenvalues with the same algebraic and
geometric multiplicities. Moreover, if $V$ is a nonsingular square matrix of appropriate dimension and $\mathcal{L}(\lambda)$ is
a strong linearization, then $V\mathcal{L}(\lambda)$ is a strong linearization as well. The matrix pencils $V
\mathcal{L}(\lambda)$ and $\mathcal{L}(\lambda)$ are usually called (strongly) equivalent.

Whenever a $kn \times kn$ matrix pencil $\mathcal{L}(\lambda)$
may be expressed as
\begin{equation} \mathcal{L}(\lambda) = \sum_{i,j=1}^k e_ie_j^T \otimes
\mathcal{L}_{ij}(\lambda) \label{def_blocksym} \end{equation}
for certain $n \times n$ matrices $\mathcal{L}_{ij}(\lambda)$, we call
$\mathcal{L}(\lambda)^{\mathcal{B}} = \sum_{i,j=1}^k e_je_i^T \otimes \mathcal{L}_{ij}(\lambda)$ the block-transpose of
$\mathcal{L}(\lambda)$ (see \cite[Definition 2.1]{HigMMT06}).
Therefore, if $\mathcal{L}(\lambda)$ of the form (\ref{def_blocksym}) satisfies
$\mathcal{L}(\lambda) = \mathcal{L}(\lambda)^{\mathcal{B}}$ it is called block-symmetric, whereas it is called
block-skew-symmetric
whenever $\mathcal{L}(\lambda) = - \mathcal{L}(\lambda)^{\mathcal{B}}.$ For the $s \times s$ leading principal submatrix of a matrix
polynomial $P(\lambda)$ we use the notation $[P(\lambda)]_s$. Using \textsc{Matlab} notation this means
$[P(\lambda)]_s = (P(\lambda))(1:s,1:s)$.

\section{Generalized Ansatz Spaces}\label{sec3}

Whenever this is not further specified, $P(\lambda) \in \mathbb{R}[\lambda]^{n \times n}$ is  a (regular or
singular)
matrix polynomial expressed in an orthogonal basis as in 
(\ref{expr_P}) with $\textnormal{deg}(P(\lambda)) = k \geq 2$. 
We make this assumption to avoid the potential occurrence
of pathological cases.
Furthermore, the main purpose of this paper is to construct linearizations for $P(\lambda)$ which is superfluous when
$P(\lambda)$ is already linear.

For $P(\lambda)$ as in (\ref{expr_P}) we define $\Phi_k(\lambda) := [ \,
\phi_{k-1}(\lambda) \; \cdots \; \phi_1(\lambda) \;
\phi_0(\lambda) \, ]^T$ and consider the set $\mathbb{M}_1(P)$ of all $kn \times kn$ matrix pencils
$\mathcal{L}(\lambda)$ satisfying
\begin{equation}  \mathcal{L}(\lambda) \big( \Phi_{k}(\lambda) \otimes I_n \big) = v \otimes P(\lambda) \label{ansatzequation}
\end{equation}
for some \enquote{ansatz vector} $v \in \mathbb{R}^k$.
For the standard monomial basis
this is just the definition of $\mathbb{L}_1(P)$ \cite[Def. 3.1]{MacMMM06} with
$\Phi_k(\lambda) = [ \, \lambda^{k-1} \; \cdots \; \lambda \; 1 \, ]^T =: \Lambda_k(\lambda).$
The same kind of generalization of $\mathbb{L}_1(P)$ to matrix polynomials in nonstandard bases has been
considered, e.g.,  in
\cite{NakNT12, TerDM09}.

 Certainly, $\mathbb{M}_1(P)$ is a vector space over $\mathbb{R}$.
Next, we introduce the
$n \times kn$ rectangular matrix pencil
$$ m_\Phi^P(\lambda) := \begin{bmatrix} \frac{(\lambda - \beta_{k-1})}{\alpha_{k-1}} P_k + P_{k-1} & P_{k-2} -
\frac{\gamma_{k-1}}{\alpha_{k-1}} P_k & P_{k-3} & \cdots & P_1 & P_0 \end{bmatrix}. $$
It is easily seen that $m_\Phi^P(\lambda)( \Phi_k(\lambda) \otimes I_n) = P(\lambda)$.
Moreover, for the $(k-1) \times k$ matrix pencil
$$ M^\star_\Phi(\lambda) = \begin{bmatrix} -\alpha_{k-2} & ( \lambda -\beta_{k-2}) & -\gamma_{k-2} & & & \\ & -\alpha_{k-3} &
( \lambda - \beta_{k-3}) &
-\gamma_{k-3}
& & \\ & & \ddots & \ddots & \ddots & \\ & & & -\alpha_1 & ( \lambda - \beta_1) & -\gamma_1 \\ & & & & -\alpha_0 & ( \lambda
- \beta_0) \end{bmatrix}  $$
we have $M_\Phi^\star(\lambda) \Phi_k(\lambda) = 0.$ Note that $M_{\Phi}(\lambda)$ depends only on the basis chosen, while $m_{\Phi}^P(\lambda)$ depends additionally on the matrix polynomial $P(\lambda)$.
Now we define
\[M_\Phi(\lambda) := M^\star_\Phi(\lambda) \otimes I_n.
\]
Certainly $M_\Phi(\lambda)( \Phi_k(\lambda) \otimes I_n) =0$ holds.
We set
\begin{equation}
F_{\Phi}^P(\lambda) := \begin{bmatrix} m_\Phi^P(\lambda) \\ M_\Phi(\lambda) \end{bmatrix} \in
\mathbb{R}[\lambda]^{kn \times kn}.
\label{stronglin_F} \end{equation}
By construction
\[
F_{\Phi}^P(\lambda) ( \Phi_k(\lambda) \otimes I_n)  = e_1 \otimes P(\lambda),
\]
thus, $F_{\Phi}^P(\lambda) \in \mathbb{M}_1(P)$ with ansatz vector
$e_1 \in \mathbb{R}^k$.
According to \cite[Thm. 2]{AmiCL09} $F_{\Phi}^P(\lambda)$ is a strong linearization for any regular $P(\lambda)$.
In \cite[Section 7]{TerDM09} it was observed that this also holds for any singular $P(\lambda)$.
 In fact, $F_{\Phi}^P(\lambda)$ may be utilized as an \enquote{anchor pencil} to construct
$\mathbb{M}_1(P)$. To this end, the next
theorem gives a concise and succinct characterization of $\mathbb{M}_1(P)$ for any matrix polynomial $P(\lambda)$
expressed in some orthogonal polynomial basis.

\begin{theorem}[Characterization of $\mathbb{M}_1(P)$] \label{thm_master1}
Let $P(\lambda)$ be an $n \times n$ regular or singular matrix polynomial of degree $k \geq 2.$ Then
$\mathcal{L}(\lambda) \in \mathbb{M}_1(P)$ with ansatz vector $v \in \mathbb{R}^k$ if
and only if
\begin{equation}  \mathcal{L}(\lambda) = \big[ \, v \otimes I_n \;~ B \, \big] F_{\Phi}^P(\lambda) \label{thm_charM1}
\end{equation}
for some matrix $B \in \mathbb{R}^{kn \times (k-1)n}$.
\end{theorem}

\begin{proof}
It is immediate that any matrix pencil $\mathcal{L}(\lambda) = [ \, v \otimes I_n \;~ B \, ]F_{\Phi}^P(\lambda)$ satisfies
(\ref{ansatzequation}) since
\begin{align*}
\big( \big[ \, v \otimes I_n \;~ B \, \big]F_{\Phi}^P(\lambda) \big) \big( \Phi_{k}(\lambda) \otimes I_n \big)
&= \big[ \, v \otimes I_n \;~ B \, \big] \big( e_1 \otimes P(\lambda) \big) \\ &= v\otimes P(\lambda).
\end{align*}
Now let $\mathcal{L}(\lambda) \in \mathbb{M}_1(P),$
thus, $ \mathcal{L}(\lambda) \big( \Phi_{k}(\lambda) \otimes I_n \big) = v \otimes P(\lambda)$ has to hold.
As $v \otimes P(\lambda) = \sum_{i=0}^k ( v  \otimes P_i \phi_i(\lambda))$
it follows that $\mathcal{L}(\lambda) \big( \Phi_{k}(\lambda) \otimes I_n \big)$ has to generate the term
$v \otimes P_k\phi_k(\lambda)$ on the right hand side of (\ref{ansatzequation}). Since $\phi_k(\lambda)$ is not an
entry of $\Phi_k(\lambda)$ and $\phi_k(\lambda)$ has degree $k$, i.e. contains a nonzero term with $\lambda^k$, we need to have $\lambda \phi_{k-1}(\lambda)$ to obtain $\lambda$ with potency $k$. To properly generate $P_k \phi_k(\lambda)$ from $\lambda \phi_{k-1}(\lambda)$ we use
the recurrence relation (\ref{def_orthbasis})
\[
v \otimes P_k \phi_k(\lambda) = v \otimes \big( \alpha_{k-1}^{-1} \big( ( \lambda - \beta_{k-1} ) \phi_{k-1}(\lambda) -
\gamma_{k-1} \phi_{k-2}(\lambda) \big) P_k \big).
\]
It gives that $\mathcal{L}(\lambda)$ may be expressed as
$$ \mathcal{L}(\lambda) = \big[ v \otimes \alpha_{k-1}^{-1} P_k \;~ \mathcal{L}_1 \, \big] \lambda + \big[ \, \ell^\star \;
 \; \mathcal{L}_0 \, \big]$$
for some matrices $
\ell^\star \in \mathbb{R}^{kn \times n}$
and $\mathcal{L}_1, \mathcal{L}_0 \in \mathbb{R}^{kn \times (k-1)n}$. Now observe that $\mathcal{L}^\star(\lambda) := [ \, v
\otimes I_n \;~ \mathcal{L}_1 \, ]F_{\Phi}^P(\lambda)$ has the form
$$ \mathcal{L}^\star(\lambda) = \big[ v \otimes \alpha_{k-1}^{-1} P_k \;~ \mathcal{L}_1 \, \big] \lambda + \big[ v \otimes
I_n \;~ \mathcal{L}_1 \, \big]F_{\Phi}^P(0)$$
as
\[
F_\Phi^P(\lambda) = F_\Phi^P(0)+
\begin{bmatrix}
\frac{\lambda}{\alpha_{k-1}}P_k & 0 & \cdots & 0\\
0 & \\
\vdots & & \lambda I_{(k-1)n}\\
0
\end{bmatrix}.
\]
Thus $\Delta \mathcal{L}(\lambda) := \mathcal{L}(\lambda) - \mathcal{L}^\star(\lambda) \in \mathbb{R}^{kn \times kn}$, i.e. it is
independent of $\lambda$. Moreover, $\Delta \mathcal{L}(\lambda)$ satisfies $\Delta \mathcal{L}(\lambda) ( \Phi_k(\lambda)
\otimes I_n) = 0$. Since $\phi_0(\lambda),
\ldots , \phi_{k-1}(\lambda), \lambda \phi_{k-1}(\lambda)$ form a basis of $\mathbb{R}_k[\lambda]$, the vector space of real
polynomials of degree $\leq k$, this implies $\Delta
\mathcal{L} = 0$ and proves that $\mathcal{L}(\lambda) = \mathcal{L}^\star(\lambda)$.
\end{proof}

In other words, Theorem \ref{thm_master1} states that\footnote{Although we confine ourselves to the case of
matrix polynomials of degree $k \geq 2$ notice that for linear matrix polynomials, $\mathbb{M}_1(P)$ simply
consists of all scalar multiples of $P(\lambda)$ itself.}
$$\mathbb{M}_1(P) = \left\lbrace \big[ \, v \otimes I_n \;~ B \, \big]F_{\Phi}^P(\lambda) \; \big| \; v \in \mathbb{R}^k, B \in
\mathbb{R}^{kn \times (k-1)n} \right\rbrace.$$
In case $\Phi_k(\lambda) = \Lambda_k(\lambda)$ denotes the monomial basis, $F_{\Phi}^P(\lambda)$ is just the first Frobenius
companion form for $P(\lambda)$ \cite[(3.1)]{MacMMM06} and $\mathbb{M}_1(P) = \mathbb{L}_1(P).$
The description of $\mathbb{L}_1(P)$ in \cite[Lem. 3.4,
Thm. 3.5]{MacMMM06} differs from (\ref{thm_charM1}) significantly although both characterizations are easily seen to be
equivalent.

Beside (\ref{ansatzequation}) we may consider its transposed version
\begin{equation}
 \big( \Phi_k(\lambda)^T \otimes I_n \big) \mathcal{L}(\lambda) = v^T \otimes P(\lambda). \label{ansatzequation2}
\end{equation}
As before, all matrix pencils satisfying (\ref{ansatzequation2}) form a vector space over $\mathbb{R}$, which we denote by
$\mathbb{M}_2(P)$. For the monomial basis $\mathbb{M}_2(P) = \mathbb{L}_2(P),$ see \cite[Def. 3.9]{MacMMM06}. It is
characterized analogously to Theorem \ref{thm_master1}.

\begin{theorem}[Characterization of $\mathbb{M}_2(P)$] \label{thm_master2}
Let $P(\lambda)$ be an $n \times n$ regular or singular matrix polynomial of degree $k \geq 2.$ Then
$\mathcal{L}(\lambda) \in \mathbb{M}_2(P)$ with ansatz vector $v \in \mathbb{R}^k$ if
and only if
\begin{equation} \mathcal{L}(\lambda) =  F_{\Phi}^P(\lambda)^{\mathcal{B}} \begin{bmatrix} v^T \otimes
I_n \\ B^{\mathcal{B}} \end{bmatrix} \label{thm_charM2} \end{equation}
for some matrix $B \in \mathbb{R}^{kn \times (k-1)n}$.
\end{theorem}

Since any pencil $\mathcal{L}(\lambda)$ of the form (\ref{thm_charM1}) or (\ref{thm_charM2}) can be uniquely identified with the
tuple $(v,B)$ we obtain the isomorphism
\begin{align*}
 \mathbb{M}_1(P) \cong \mathbb{R}^k \times
\mathbb{R}^{kn \times (k-1)n} \cong \mathbb{M}_2(P).
\end{align*}
This isomorphism was also observed in the proof of \cite[Thm. 4.4]{TerDM09} in the context of matrix polynomials in the monomial basis.

\begin{corollary} \label{cor_dimension}
 For any $n \times n$ regular or singular matrix polynomial $P(\lambda)$ of degree $k$
$$\textnormal{dim} \, \mathbb{M}_1(P) = \textnormal{dim} \, \mathbb{M}_2(P) = k(k-1)n^2 +k.$$
\end{corollary}
Corollary \ref{cor_dimension} is essentially \cite[Cor. 3.6]{MacMMM06} for the monomial basis.
We now give a universal
linearization condition for matrix pencils in $\mathbb{M}_1(P)$ and $\mathbb{M}_2(P)$ that does not depend on
the chosen basis at all.

\begin{corollary} \label{cor_lincondition1}
 Let $P(\lambda)$ be an $n \times n$ regular or singular matrix polynomial of degree $k \geq 2$ and $\mathcal{L}(\lambda) \in
\mathbb{M}_1(P)$ of the
form (\ref{thm_charM1}) or $\mathcal{L}(\lambda) \in \mathbb{M}_2(P)$ of the form (\ref{thm_charM2}). Then the following statements hold:
\begin{enumerate}
 \item The matrix pencil $\mathcal{L}(\lambda)$ is a strong linearization for $P(\lambda)$ if
 \begin{equation} \textnormal{rank} \big( \big[ v \otimes I_n \;~ B \, \big] \big) = kn \label{cor_lincondition}
\end{equation}
 holds regardless whether $P(\lambda)$ is regular or singular. Certainly, (\ref{cor_lincondition}) is equivalent to $\big[ v \otimes I_n \;~ B \, \big] \in
\textnormal{GL}_{kn}( \mathbb{R})$.
\item If $P(\lambda)$ is a regular matrix polynomial and $\mathcal{L}(\lambda)$ a linearization for $P(\lambda)$, then
the rank condition (\ref{cor_lincondition}) is satisfied.
 \end{enumerate}
\end{corollary}

Notice that any pencil $\mathcal{L}(\lambda) \in \mathbb{M}_1(P)$ or $\mathcal{L}(\lambda) \in \mathbb{M}_2(P)$
that does not satisfy the condition (\ref{cor_lincondition}) is automatically singular.

\begin{proof}
\begin{enumerate}
\item Whenever $\textnormal{rank} ( [ v \otimes I_n \;~ B \, ] ) = kn$, $\mathcal{L}(\lambda) = [
\, v \otimes I_n \;~ B ]F_{\Phi}^P(\lambda)$ is strongly equivalent to $F_{\Phi}^P(\lambda)$ and thus a strong linearization
for $P(\lambda)$.
\item If $\textnormal{rank} ( [ v \otimes I_n \;~ B \, ] ) < kn$, $\mathcal{L}(\lambda)$ is
singular and therefore not a linearization for any regular $P(\lambda)$.
\end{enumerate} \vspace*{-0.8cm}
\end{proof}

Corollary \ref{cor_lincondition1}.1 is essentially just a reformulation of
\cite[Thm. 4.1]{MacMMM06} for the monomial basis $\Phi = \Lambda = \lbrace \lambda^j \rbrace_{j=0}^{\infty}$.
To see this, assume that $\mathcal{L}(\lambda) \in \mathbb{L}_1(P)$ is as
in (\ref{thm_charM1}) and notice that $F_{\Phi}^P(\lambda) = F_{\Lambda}^P(\lambda)$ is simply the first Frobenius companion form for $P(\lambda)$. Now let $M \in \mathbb{R}^{k \times k}$ be a nonsingular matrix that satisfies $Mv=e_1$. Premultiplying $\mathcal{L}(\lambda)$ with $M \otimes I_n$ yields
$$ \begin{aligned} \mathcal{L}^\star(\lambda)= \big( M \otimes I_n \big) \big[ \, v \otimes I_n \;~ B \, \big]
F_{\Lambda}^P(\lambda) &= \left[\begin{array}{c|c} e_1 \otimes I_n & (M \otimes I_n)B \end{array} \right]
F_{\Lambda}^P(\lambda) \\ &=: \left[ \begin{array}{c|c} I_n &  B_{11}^\star \\  \hline 0 &
B_{21}^\star  \end{array} \right]F_{\Lambda}^P(\lambda) \end{aligned} $$
so $\mathcal{L}^\star(\lambda) \in \mathbb{L}_1(P)$ with ansatz vector $e_1$. Now
$$
\mathcal{L}^\star(\lambda) = \left[ \begin{array}{c|c} P_k & B_{11}^\star \\ \hline
0 & B_{21}^\star \end{array} \right] \lambda + \left[ \begin{array}{c} m_{\Lambda}^P(0) \\ \hline (M \otimes
I_n)BM_{\Lambda}(0) \end{array} \right].
$$
In the form given above $\mathcal{L}^\star(\lambda)$ corresponds to equation (4.2) in \cite{MacMMM06}.
It is said that $\mathcal{L}(\lambda)$ has full $Z$-rank whenever $B_{21}^\star$ has full rank for
any chosen nonsingular matrix $M$ with the property $Mv=e_1$ \cite[Thm. 4.1, Def. 4.3]{TerDM09}. This
is the case if
and only if $[ \, e_1 \otimes I_n \;~ (M \otimes I_n)B \, ]$ has full rank. Since
$$ \big[ \, e_1 \otimes I_n \;~ (M \otimes I_n)B \, \big] = (M \otimes I_n) \big[ \, v \otimes I_n \;~ B \,
\big] $$ and as $M \otimes I_n$ is nonsingular as well, $[ \, e_1 \otimes I_n \;~ (M \otimes I_n)B \, ]$ has full
rank
if and only if $[ \,  v \otimes I_n \;~ B \, ]$ has full rank. We summarize this observation in the next corollary.

\begin{corollary} \label{cor_fullZrank}
Let $P(\lambda)$ be an $n \times n$ matrix polynomial of degree $k \geq 2$ and assume $\mathcal{L}(\lambda) \in \mathbb{L}_1(P)$ is given as
$$ \mathcal{L}(\lambda) = \big[ \, v \otimes I_n \;~ B \, \big] F_{\Lambda}^P(\lambda).$$
Then $\mathcal{L}(\lambda)$ has full $Z$-rank if and only if $\textnormal{rank}([ \, v \otimes I_n \;~ B \, ]) = kn$.
\end{corollary}

Moreover, it can be easily checked that the $Z$-rank-deficiency of a pencil $\mathcal{L}(\lambda) \in \mathbb{L}_1(P)$ carries over to the matrix $[ \,  v \otimes I_n \;~  B \, ]$, i.e. if the $Z$-rank of $\mathcal{L}(\lambda)$ is $s < (k-1)n,$ so its $Z$-rank-deficiency is $t = (k-1)n - s,$ then it follows that $\textnormal{rank}([ \, v \otimes I_n \;~ B \, ]) = kn - t$. Thus, there is in fact no loss of information in considering the rank of the matrix $[ \, v \otimes I_n \;~ B \, ]$ instead of the $Z$-rank.

For nonmonomial bases the linearization condition for pencils in $\mathbb{M}_1(P)$ presented in \cite[Prop. 4.9]{MacMMM05} requires even more work. First the pencil $\mathcal{L}(\lambda)$ has to be transformed into an element of $\mathbb{L}_1(P)$ via a basis change and then, in a second step, it has to be expressed with the
ansatz vector $e_1$ (according to the discussion above) to compute the $Z$-rank.
Fortunately, we may apply Corollary \ref{cor_lincondition1} to the pencils in $\mathbb{M}_1(P)$
(expressed as in (\ref{thm_charM1})) right away without a change of the ansatz vector or the polynomial basis.
Even if the pencil $\mathcal{L}(\lambda) \in \mathbb{M}_1(P)$ is expressed as $\mathcal{L}(\lambda) = X \lambda +
Y$, $v$ and $B$ may easily be recovered to check (\ref{cor_lincondition}) since $\mathcal{L}(\lambda) = [ \, v
\otimes I_n \;~ B \, ]F_{\Phi}^P(\lambda)$ may also be expressed as
$$\mathcal{L}(\lambda) = \big[ \,  v \otimes \alpha_{k-1}^{-1}P_k \;~ B \, ] \lambda + \mathcal{L}(0)$$
that is $v$ and $B$ appear directly in the matrix $X$.

The construction of strong linearizations for
matrix polynomials expressed in the Chebyshev basis proposed in \cite{LawP16} gets along without such conditions.
Recall that in \cite[Thm. 2.1]{NakNT12} the Strong Linearization Theorem from \cite[Thm. 4.3]{MacMMM06} was revisited and proven for all generalized ansatz spaces considering any degree-graded basis (these have been introduced in \cite[Sec. 4.2]{MacMMM06}).
In fact, all three equivalent conditions given there are equivalent
to (\ref{cor_lincondition}) for regular matrix polynomials expressed in orthogonal bases.
This can easily be seen by Corollary \ref{cor_lincondition1}. We state the Strong Linearization Theorem \cite[Thm. 2.1]{NakNT12} according to our discussion adding the equivalent condition from Corollary \ref{cor_lincondition1}.

\begin{theorem}[Strong Linearization Theorem] \label{thm_stronglin}
Let $P(\lambda)$ be an $n \times n$ regular matrix polynomial of degree $k \geq 2$ and $\mathcal{L}(\lambda) \in \mathbb{M}_1(P)$ as given in (\ref{thm_charM1}). Then the following statements are equivalent:
\begin{enumerate}
\item $\mathcal{L}(\lambda)$ is a linearization for $P(\lambda)$
\item $\mathcal{L}(\lambda)$ is a regular matrix pencil.
\item $\mathcal{L}(\lambda)$ is a strong linearization for $P(\lambda).$
\item $\textnormal{rank}([ \, v \otimes I_n \;~ B \, ]) = kn$ (i.e. $\mathcal{L}(\lambda)$ has full $Z$-rank, see Corollary \ref{cor_fullZrank})
\end{enumerate}
\end{theorem}

For a discussion of linearizations for singular polynomials in non-monomial bases see
\cite[Sec. 7]{TerDM09}.
Since almost every matrix of the form $[ \, v \otimes I_n \;~ B \, ]$ has full rank, we obtain the following genericity
statement. This result was already stated in \cite[Thm. 4.7]{MacMMM06} for matrix polynomials $P(\lambda)$ in the monomial basis.

\begin{corollary} \label{cor_generic}
For any $n \times n$ regular or singular matrix polynomial $P(\lambda)$ almost every matrix pencil in
$\mathbb{M}_1(P)$ is a strong linearization for $P(\lambda)$.
\end{corollary}

Certainly, an analogous statement to Corollary \ref{cor_generic} holds for $\mathbb{M}_2(P)$.

\section{The Recovery of right Eigenvectors}
\label{sec:eigenvectorrecovery}

We now show how eigenvectors for regular $P(\lambda)$ as in (\ref{expr_P}) may be recovered from eigenvectors of linearizations in
$ \mathbb{M}_1(P)$. The main ideas behind this derivation follow mainly the approach in \cite{MacMMM06} and
\cite[Sec. 5]{TerDM09}. However, at first we show that we can restrict the study of eigenvectors essentially to
$F_{\Phi}^P(\lambda).$ This yields a new kind of linearization condition for pencils in $\mathbb{M}_1(P)$ and
$\mathbb{M}_2(P)$ respectively.
\begin{proposition} \label{prop_eigrec1}
Let $P(\lambda)$ be an $n \times n$ regular matrix polynomial of degree $k \geq 2$ and $\mathcal{L}(\lambda) \in
\mathbb{M}_1(P)$. Then the following statements hold:
\begin{enumerate}
 \item Every right eigenvector of $F_{\Phi}^P(\lambda)$
(for any eigenvalue) is a right eigenvector of $\mathcal{L}(\lambda).$
\item Let $\mathcal{L}(\lambda)$ be a (strong) linearization for $P(\lambda)$. Then every right eigenvector of  $\mathcal{L}(\lambda)$ (for any eigenvalue) is a right eigenvector of $F_{\Phi}^P(\lambda).$
\end{enumerate}
\end{proposition}

\begin{proof}
\begin{enumerate}
\item  This is clear since $\mathcal{N}_r( F_{\Phi}^P(\alpha)) \subseteq \mathcal{N}_r( \mathcal{L}(\alpha))$
always holds, see (\ref{thm_charM1}).
\item   Whenever $\mathcal{L}(\lambda)$ is a (strong) linearization for $P(\lambda)$, we obtain from 1. that
$\mathcal{N}_r( F_{\Phi}^P(\alpha)) = \mathcal{N}_r( \mathcal{L}(\alpha))$
has to hold. Thus every right eigenvector of $\mathcal{L}(\lambda)$ is a right eigenvector of
$F_{\Phi}^P(\lambda)$.
\end{enumerate}\vspace*{-0.8cm}
\end{proof}

From Proposition \ref{prop_eigrec1} we directly obtain the following linearization condition for matrix pencils in $\mathbb{M}_1(P)$:
\begin{proposition} \label{prop_eigrec2}
Let $P(\lambda)$ be an $n \times n$ regular matrix polynomial of degree $k \geq 2$ and $\mathcal{L}(\lambda) \in
\mathbb{M}_1(P)$. Then $\mathcal{L}(\lambda)$ is a
strong linearization for $P(\lambda)$ if and only if every right eigenvector of $\mathcal{L}(\lambda)$ is a right
eigenvector of $F_{\Phi}^P(\lambda)$ for some (finite or infinite) eigenvalue.
\end{proposition}

\begin{proof}
$ \Rightarrow$ Suppose $\alpha$ is an eigenvalue of $\mathcal{L}(\lambda)$ and $u
\in
\mathbb{C}^{kn}$ is the corresponding eigenvector, i.e. $\mathcal{L}(\alpha)u=0$. Then, since
$$\mathcal{L}(\alpha) = \big[ \,  v \otimes I_n \;~ B \, ]F_{\Phi}^P(\alpha)$$
assuming that $F_{\Phi}^P(\alpha)u \neq 0$ we obtain that $F_{\Phi}^P(\alpha)u \in \mathcal{N}_r( [ \, v \otimes
I_n \;~ B \, ])$. Thus, $[ \,  v \otimes I_n \;~ B \, ]$ is singular and $\mathcal{L}(\lambda)$ is no
linearization for $P(\lambda)$, a contradiction.
$\Leftarrow$ On the other hand, assume that $[ \,  v \otimes I_n \;~ B \, ]$ is singular, i.e.
$\mathcal{L}(\lambda)$ is not a linearization for $P(\lambda)$. Then there exists some $w \in \mathbb{C}^{kn}$
such that $[ \,  v \otimes I_n \;~ B \, ]w=0$. Now take any $\beta \in \mathbb{C}$ that is not an eigenvalue of
$P(\lambda)$, then $F_{\Phi}^P(\beta)$ is nonsingular. Therefore we may solve $F_{\Phi}^P(\beta)z = w$ for $z$ and
thus $z$ is a right eigenvector of $\mathcal{L}(\beta)$ that is not an eigenvector of $F_{\Phi}^P(\beta)$.
\end{proof}

Note that Proposition \ref{prop_eigrec1} states a linearization condition for any pencil $\mathcal{L}(\lambda)$ in
$\mathbb{M}_1(P)$ in terms of (a comparison of) the right eigenvectors of $\mathcal{L}(\lambda)$ and
$F_{\Phi}^P(\lambda)$.\footnote{Another linearization condition based upon left eigenvectors is derived in
Section \ref{sec:eigenvectorexclusion}.} Certainly, a similar statement holds for pencils in $\mathbb{M}_2(P)$.

The following proposition shows how eigenvectors of regular $P(\lambda)$ can be recovered from eigenvectors of
linearizations in $\mathbb{M}_1(P).$ This has already been observed in a slightly different form in \cite[Sec. 7]{TerDM09}.
It can be proven exactly analogous to \cite[Thm. 3.8, Thm. 3.14, Thm. 4.4]{MacMMM06}. Taking Proposition \ref{prop_eigrec1} and Proposition \ref{prop_eigrec2} into account, Proposition 3 allows the complete
eigenvector recovery for linearizations in $\mathbb{M}_1(P)$ and $\mathbb{M}_2(P)$.

\begin{proposition} \label{prop_eigenvectorrecovery}
 Let $P(\lambda)$ be an $n \times n$ regular matrix polynomial of degree $k \geq 2$. Then the following statements hold:
\begin{enumerate}
\item Let $\alpha$ be some finite eigenvalue of $P(\lambda)$. Then $u \in \mathcal{N}_r(P(\alpha))$ if and only if
$\Phi_k(\alpha) \otimes u \in \mathcal{N}_r(F_{\Phi}^P(\alpha))$. Moreover, every right eigenvector $w$ of
$F_{\Phi}^P(\alpha)$ has the form
$w=\Phi_k(\alpha) \otimes u$ for some $u \in \mathcal{N}_r(P(\alpha))$.
\item Let $\alpha$ be infinity. Then $u \in \mathcal{N}_r( \textnormal{rev}_k \, P(0))$ if and only if
$e_1 \otimes u \in \mathcal{N}_r(\textnormal{rev}_1 \, F_{\Phi}^P(0))$. Moreover, every right eigenvector $w$ of
$\textnormal{rev}_1 \, F_{\Phi}^P(0)$ has
the form
$w=e_1 \otimes u$ for some $u \in \mathcal{N}_r(\textnormal{rev}_k \, P(0))$.
\end{enumerate}
\end{proposition}

It is well-known that for singular matrix polynomials recovering the complete eigenstructure
comprises not only of the finite and infinite eigenvalues
but also the left and right minimal indices and minimal bases. Without further ado we would like to
point the reader to \cite[Sec. 7]{TerDM09}.

\section{A Note on Singular Matrix Polynomials}
\label{sec:singular}

As already discussed in Section \ref{sec3}, when $P(\lambda)$ is regular, any linearization in $\mathbb{L}_1(P)$
(or $\mathbb{M}_1(P)$) is necessarily a strong linearization.
In \cite[Ex. 3]{TerDM09} is was shown that the equivalence of strong linearizations and linearizations does not
hold for singular $P(\lambda).$
 Moreover, \cite[Ex. 2]{TerDM09} shows that the condition (\ref{cor_lincondition}) turns out to be
neither necessary for linearizations nor for strong linearizations. In this section we consider singular
matrix polynomials $P(\lambda)$ and give a sufficient condition on when the equivalence of being a linearization, a strong linearization and having full $Z$-rank holds. The main result of this section is the following theorem which extends \cite[Lem. 5.5]{TerDM09} by complementing it to an equivalence statement. Moreover, it is extended to orthogonal bases.

\begin{theorem} \label{thm_eetsingular}
 Let $P(\lambda)$ be an $n \times n$ singular matrix polynomial of degree $k \geq 2$ and assume
$\mathcal{L}(\lambda) \in
\mathbb{M}_1(P)$ as in (\ref{thm_charM1}). Then $\textnormal{rank}([ \, v \otimes I_n \;~ B \, ])=kn$
if and only if $$ u(\lambda)^T(v \otimes I_n) \neq 0$$
for every $u(\lambda) \in \mathcal{N}_\ell(\mathcal{L}(\lambda)).$

\end{theorem}

\begin{proof}
Assume that $\mathcal{L}(\lambda)$ as given in (\ref{thm_charM1}) satisfies $\textnormal{rank}([ \, v \otimes I_n \;~ B \, ])=kn$ and let $ 0 \neq u(\lambda) \in \mathcal{N}_\ell(\mathcal{L}(\lambda))$.
Defining $w(\lambda) \in \mathbb{R}(\lambda)^{kn}$ as
 $$w(\lambda)^T = \big[ \, w_1(\lambda) \; w_2(\lambda) \; \cdots \; w_{kn}(\lambda) \, \big] := u(\lambda)^T\big[ \, v \otimes I_n \;~ B \, \big]$$
 and assuming that $u(\lambda)^T( v \otimes I_n) = 0$, we obviously obtain $$[ \, w_1(\lambda) \; \cdots \; w_n(\lambda) \, ] = 0.$$ Moreover,
since
$u(\lambda) \in \mathcal{N}_\ell(\mathcal{L}(\lambda))$, we have
 \begin{equation} u(\lambda)^T \mathcal{L}(\lambda) = u(\lambda)^T [ \, v \otimes I_n \;~ B \, ]F_{\Phi}^P(\lambda) = w(\lambda)^T
F_{\Phi}^P(\lambda)=0. \label{exclusion1} \end{equation}
Using the fact that $[ \, w_1(\lambda) \; \cdots \; w_n(\lambda) \, ] = 0$, (\ref{exclusion1}) and the block-Hessenberg structure of
$F_{\Phi}^P(\lambda)$ imply $\alpha_{k-2}[
w_{n+1}(\lambda) \; \cdots \; w_{2n}(\lambda) \, ] = [ \, 0 \; \cdots \; 0 \, ]$, thus $w_{n+1}(\lambda) = \cdots = w_{2n}(\lambda) =0$. Therefore we actually have
\begin{equation}
\big[ \, 0 \; 0 \;
\ldots \; 0 \; 0 \; w_{2n+1}(\lambda) \; \cdots \; w_{kn}(\lambda) \, \big] F_{\Phi}^P(\lambda) = \big[ \, 0 \; \cdots \; 0 \, \big].
\label{exclusion2} \end{equation}
From (\ref{exclusion2}) the same observation yields $\alpha_{k-3}[ w_{2n+1}(\lambda) \; \cdots \; w_{3n}(\lambda) \, ] = [ \, 0
\; \cdots \; 0 \,
]$ implying $w_{2n+1}(\lambda) = \cdots = w_{3n}(\lambda) =0$. Continuing this procedure up to $\alpha_0$ we obtain $w(\lambda) \equiv 0$. In other words,
$u(\lambda) \in \mathcal{N}_\ell([ \, v \otimes I_n \;~ B ])$. This implies $[ \, v \otimes I_n
\;~ B \, ]$ to be singular. Since we have assumed $[ \, v \otimes I_n \;~ B \, ]$ to have full rank, $u(\lambda)^T(v \otimes I_n) = 0$ implies $u(\lambda) \equiv 0$, a contradiction. Thus, the assumption $u(\lambda)^T (v \otimes I_n) = 0$ must have been false and we have $u(\lambda)(v \otimes I_n) \neq 0$ for every $0 \neq u(\lambda) \in \mathcal{N}_\ell(\mathcal{L}(\lambda))$.

Now suppose $u(\lambda)^T(v \otimes I_n) \neq 0$ holds for every $0 \neq u(\lambda) \in \mathcal{N}_\ell(\mathcal{L}(\lambda))$. Assuming $\textnormal{rank}([ \, v \otimes I_n \;~ B \, ]) < kn$ implies the existence of at least one vector $0 \neq q \in \mathbb{R}^{kn}$ with $q^T[ \,  v \otimes I_n \;~ B \, ] = 0$. Since $\mathcal{L}(\lambda) = [ \, v \otimes I_n \;~ B \, ]F_{\Phi}^P(\lambda)$ we have $q^T \mathcal{L}(\lambda) = 0$, so obviously $q \in \mathcal{N}_\ell( \mathcal{L}(\lambda))$. Now in particular $q$ satisfies $$q^T(v \otimes I_n) = 0$$ which is a contradiction for we assumed $u(\lambda)^T(v \otimes I_n) \neq 0$ for every $0 \neq u(\lambda) \in \mathcal{N}_\ell(\mathcal{L}(\lambda))$. Thus we must have $\textnormal{rank}([ \, v \otimes I_n \;~ B \, ]) = kn.$
\end{proof}

We obtain an immediate corollary:

\begin{corollary} \label{singular_lin}
Let $P(\lambda)$ be an $n \times n$ singular matrix polynomial of degree $k \geq 2$ and assume $\mathcal{L}(\lambda) \in \mathbb{M}_1(P)$. If
$$u(\lambda)^T(v \otimes I_n) \neq 0$$
for all $u(\lambda) \in \mathcal{N}_\ell(\mathcal{L}(\lambda))$ then $\mathcal{L}(\lambda)$ is a strong linearization for $P(\lambda)$.
\end{corollary}

\begin{proof}
This follows immediately from Theorem \ref{thm_eetsingular} since $u(\lambda)^T(v \otimes I_n) \neq 0$
for all $u(\lambda) \in \mathcal{N}_\ell(\mathcal{L}(\lambda))$ implies $\textnormal{rank}([ \, v \otimes I_n \;~ B \, ])=kn.$
This in turn implies $\mathcal{L}(\lambda)$ to be a strong linearization for $P(\lambda)$ according to Corollary \ref{cor_lincondition1}.
\end{proof}

We now state a modified version of the \textit{Strong Linearization Theorem} adapted for singular matrix polynomials. The original theorem applies to regular matrix polynomials and was proven for $\mathbb{L}_1(P)$ in \cite[Thm. 4.3]{MacMMM06} and extended to degree-graded bases in \cite[Thm. 2.1]{NakNT12}.

\begin{theorem}[Strong Linearization Theorem] \label{thm_stronglinsing}
Let $P(\lambda)$ be an $n \times n$ singular matrix polynomial of degree $k \geq 2$ and $\mathcal{L}(\lambda) \in \mathbb{M}_1(P)$ as given in (\ref{thm_charM1}). Additionally assume that \begin{equation} u(\lambda)^T(v \otimes I_n) \neq 0 \label{stronglinthm_sing} \end{equation} for all $0 \neq u(\lambda) \in \mathcal{N}_\ell(\mathcal{L}(\lambda))$. Then the following statements are equivalent:
\begin{enumerate}
\item $\textnormal{rank}([ \, v \otimes I_n \;~ B \, ]=kn$ (i.e. $\mathcal{L}(\lambda)$ has full $Z$-rank, see Corollary \ref{cor_fullZrank}).
\item $\mathcal{L}(\lambda)$ is a strong linearization for $P(\lambda)$.
\item $\mathcal{L}(\lambda)$ is a linearization for $P(\lambda)$.
\end{enumerate}
\end{theorem}

\begin{proof}
It is clear that $1. \Rightarrow 2. \Rightarrow 3.$ holds even without the assumption $u(\lambda)^T(v \otimes I_n) \neq 0$ for all $0 \neq u(\lambda) \in \mathcal{N}_\ell(\mathcal{L}(\lambda))$ and that $3. \Rightarrow 1.$ follows from Theorem \ref{thm_eetsingular} taking (\ref{stronglinthm_sing}) into account.
\end{proof}

\section{Double Generalized Ansatz Spaces and Block-Symmetry}\label{sec4}

In this section, we characterize matrix pencils that are contained in both generalized ansatz spaces
$\mathbb{M}_1(P)$ and $\mathbb{M}_2(P)$ for an $n \times n$ matrix polynomial of degree $k \geq 2$.

Certainly, if some matrix pencil $\mathcal{L}(\lambda)$ satisfies (\ref{ansatzequation}),
$\mathcal{L}(\lambda)^{\mathcal{B}}$ satisfies
(\ref{ansatzequation2}) and vice versa. Consequently, if $\mathcal{L}(\lambda) = \mathcal{L}(\lambda)^{\mathcal{B}}$,
$\mathcal{L}(\lambda) \in \mathbb{M}_1(P) \cap \mathbb{M}_2(P)$.
Thus, the vector space $ \mathbb{DM}(P) :=
\mathbb{M}_1(P) \cap \mathbb{M}_2(P)$, called \enquote{double generalized ansatz space} in the following, contains
all block-symmetric pencils from $\mathbb{M}_1(P)$ and $\mathbb{M}_2(P)$.
Similarly, in the monomial case, the double ansatz space $\mathbb{DL}(P) = \mathbb{L}_1(P) \cap \mathbb{L}_2(P)$
contains all block-symmetric pencils from $\mathbb{L}_1(P),$ see \cite{HigMMT06}.
We now give a rather surprising statement on block-skew-symmetric pencils in $\mathbb{M}_1(P)$.

\begin{proposition} \label{prop_skewsym}
Let $P(\lambda)$ be an $n \times n$ regular or singular matrix polynomial of degree $k \geq 2$ and let $\mathcal{L}(\lambda) \in \mathbb{M}_1(P)$ be block-skew-symmetric.
Then $\mathcal{L}(\lambda)$ satisfying (\ref{ansatzequation}) with $v = [ \, 0 \; \, v_2 \; \, v_3 \; \, \cdots \; \, v_k \, ]^T
\in \mathbb{R}^k$ implies $\mathcal{L}(\lambda) \equiv 0$.
\end{proposition}

\begin{proof}
Let
\begin{align} \mathcal{L}(\lambda) &= [ \, v \otimes I_n \;~ B \, ]F_{\Phi}^P(\lambda) = [ \, v \otimes
\alpha_{k-1}^{-1}P_k \;~ B]
\lambda + [ \, v \otimes I_n \;~ B \, ]F_{\Phi}^P(0) \notag \\
&= - \begin{bmatrix} v^T \otimes \alpha_{k-1}^{-1}P_k \\ B^{\mathcal{B}} \end{bmatrix} \lambda - F_{\Phi}^P(0)^{\mathcal{B}}
\begin{bmatrix} v^T \otimes I_n \\ B^{\mathcal{B}} \end{bmatrix}
\label{blocksym1} \end{align}
be block-skew-symmetric and assume $v = [ \, 0 \; \, v_2 \; \, v_3 \; \, \cdots \; \, v_k \, ]^T$. Regarding (\ref{blocksym1}),
the
block-skew-symmetry of
$\mathcal{L}(\lambda)$ a priori implies $B$
  to have the form $$ B = \left[ \begin{array}{c} Z \\ B^\star \end{array} \right]$$
with $Z = [ \, v_2 \; \, v_3 \; \, \cdots \; \, v_k] \otimes (-\alpha_{k-1}^{-1}P_k) \in \mathbb{R}^{n \times (k-1)n}$ and a
block-skew-symmetric $(k-1)n \times (k-1)n$ matrix $B^\star$. Let $B^\star = [B^\star_{i,j}]_{i,j=1}^{k-1}$ with $B^\star_{i,j}
\in \mathbb{R}^{n
\times n}$. The block-skew-symmetry then
implies $B^\star_{j,j}=0_n$ for all $j = 1, \ldots , k-1$. Now we consider the leading principal submatrices of
$\mathcal{L}(\lambda)$ which certainly all have to be block-skew-symmetric.

Since $[\mathcal{L}(\lambda)]_n = \alpha_{k-2} \alpha_{k-1}^{-1} v_2P_k = 0$, we have $v_2=0$. Now choose an
index $2 \leq i \leq k-1$ and assume $v_1 = v_2 = \cdots = v_i=0$ and $[B^\star]_{(i-1)n}=0$.\footnote{Notice that these
conditions are satisfied for $i=2$.} Then the $in \times in$ leading
principal submatrix $\big[ \mathcal{L}(\lambda) \big]_{in}$ of $\mathcal{L}(\lambda)$ takes in absolute value $\left| [
\mathcal{L}(\lambda) ]_{in} \right|$ the form
$$ \left| \big[ \mathcal{L}(\lambda) \big]_{in} \right| =   \left| \left[ \begin{array}{ccc|c} 0 & \cdots & 0 &  \alpha_{k-1-i}
\alpha_{k-1}^{-1} v_{i+1}P_k  \\ \vdots & & \vdots &  \alpha_{k-1-i}B_{1,i} \\ \vdots & & \vdots & \vdots \\ 0 & \cdots & 0 &
\alpha_{k-1-i}B_{i-1,i}
 \end{array} \right] \right| .  $$
Since $[ \mathcal{L}(\lambda) ]_{in}$ is block-skew-symmetric it follows that $B_{1,i}= \cdots = B_{i-1,i}=0_n$ and in
particular $v_{i+1}=0$. Therefore we have shown that $v_1 = \cdots = v_{i+1}=0$ and that $[B^\star]_{in} =0$.
Inductively, $i=k-1$
yields $v=0$ and
$B^\star=0$.
\end{proof}

Using Proposition \ref{prop_skewsym} we assume $P(\lambda)$ to be an arbitrary $n \times n$ matrix polynomial and
obtain a simple proof of the following theorem.

\begin{theorem}\label{theo3}
Let $P(\lambda)$ be an $n \times n$ regular or singular matrix of degree $k \geq 2$. Then any matrix pencil $\mathcal{L}(\lambda) \in \mathbb{DM}(P)$ is
block-symmetric.
\end{theorem}

\begin{proof}
  Let $\mathcal{L}(\lambda) \in \mathbb{DM}(P)$. Then $\mathcal{L}(\lambda)$ can be expressed as
  \begin{align*} \mathcal{L}(\lambda) &= v \otimes m_\Phi^P(\lambda) + B_1M_\Phi(\lambda) \\ &= w^T \otimes m_\Phi^P(\lambda)^{\mathcal{B}} +
M_\Phi(\lambda)^{\mathcal{B}}B_2^{\mathcal{B}}
\end{align*}
  as an element of $\mathbb{M}_1(P)$ and $\mathbb{M}_2(P)$ respectively. Regarding $\mathcal{L}(\lambda)$ in the form
$\mathcal{L}(\lambda) = X \lambda + Y$ this shows that $[X]_n = v_1 \alpha_{k-1}^{-1}P_k = w_1 \alpha_{k-1}^{-1}P_k$. Thus it
follows that $v_1 =
w_1$. Now note that $\mathcal{L}(\lambda)$ (seen as an element of $\mathbb{M}_2(P)$) via block-transposition becomes an element
  of $\mathbb{M}_1(P)$. Therefore
 \begin{align*}
 \widetilde{\mathcal{L}}(\lambda) := \mathcal{L}(\lambda) - \mathcal{L}(\lambda)^{\mathcal{B}} &= (v-w) \otimes
m_\Phi^P(\lambda) + (B_1-B_2)M_\Phi(\lambda) \\ &=: \widetilde{v} \otimes m_\Phi^P(\lambda) + \widetilde{B}M_\Phi(\lambda)
 \end{align*}
is a block-skew-symmetric pencil in $\mathbb{M}_1(P)$. Since $\widetilde{v} = [ \, 0 \; \, \widetilde{v}_2 \; \, \widetilde{v}_3
\; \, \cdots \; \, \widetilde{v}_k \, ]^T$,
applying Proposition \ref{prop_skewsym} to $\widetilde{\mathcal{L}}(\lambda)$ we obtain $\mathcal{L}(\lambda) =
\mathcal{L}(\lambda)^{\mathcal{B}}$.
\end{proof}

Following the previous proof we obtain the next result on the explicit form of pencils in the double generalized
ansatz space.

\begin{corollary} \label{cor_DMpencils}
Let $P(\lambda)$ be an $n \times n$ regular or singular matrix polynomial of degree $k \geq 2$ and
 $$ \mathcal{L}(\lambda) = \big[ \, v \otimes I_n \;~ B_1 \, ]F_{\Phi}^P(\lambda) =
F_{\Phi}^P(\lambda)^{\mathcal{B}} \begin{bmatrix} w^T \otimes I_n \\ B_2 \end{bmatrix} \in \mathbb{DM}(P).$$
Then $v=w$ and $B_1 = B_2^\mathcal{B}$.
\end{corollary}

Do not overlook that pencils in $\mathbb{DM}(P)$ not only have to have equal left and right ansatz
vectors. Corollary \ref{cor_DMpencils} makes a stronger statement. In fact, the matrices $B_1$ and $B_2$ are
additionally related to each other as in a way that $B_1 = B_2^\mathcal{B}$.

Recalling that any block-symmetric matrix pencil $\mathcal{L}(\lambda) = \mathcal{L}(\lambda)^{\mathcal{B}}$ from
$\mathbb{M}_1(P)$ is in $\mathbb{DM}(P)$ we obtain
\[
 \mathbb{DM}(P) = \big\lbrace \mathcal{L}(\lambda) \in \mathbb{M}_1(P) \; \big| \; \mathcal{L}(\lambda) =
\mathcal{L}(\lambda)^{\mathcal{B}} \big\rbrace.
\]

Clearly, also all pencils in $\mathbb{DL}(P)$ for matrix polynomials $P(\lambda)$
in monomial basis are block-symmetric. This has first been proven in \cite{HigMMT06}. At the end of this section we show that a result similar to \cite[Thm. 6.1]{TerDM09} holds for the generalized ansatz space $\mathbb{DM}(P)$. In particular, we may restrict the study of $\mathbb{DM}(P)$ to regular matrix polynomials due to the following theorem.

\begin{theorem} \label{nolinearization}
Let $P(\lambda)$ be an $n \times n$ singular matrix polynomial. Then none of the pencils in $\mathbb{DM}(P)$ is a linearization for $P(\lambda)$.
\end{theorem}
\begin{proof} The proof follows exactly the same argumentation as that of \cite[Thm. 6.1]{TerDM09}. Assume $\mathcal{L}(\lambda) \in \mathbb{DM}(P)$ with $\textnormal{rank}([ \, v \otimes I_n \;~ B \, ])=kn$. According to \cite[Th. 7.2]{TerDM09}, seeing $\mathcal{L}(\lambda)$ as an element of $\mathbb{M}_1(P)$, the right minimal indices of $\mathcal{L}(\lambda)$ are
$$ (k-1) + \epsilon_1 \leq (k-1) + \epsilon_2 \leq \cdots \leq (k-1) + \epsilon_p$$
if the right minimal indices of $P(\lambda)$ are $\epsilon_1 \leq \epsilon_2 \leq \cdots \leq \epsilon_p$. This leads to a contradiction with Theorem \cite[Thm. 7.3]{TerDM09} interpreting $\mathcal{L}(\lambda)$ as an element of $\mathbb{M}_2(P)$, \footnote{Note that the authors of \cite{TerDM09} restricted Section 7 to the study of the \textit{right ansatz} but emphasize that analogous results hold for the dual \textit{left ansatz}. In the proof of Theorem \ref{nolinearization} we use these results even though they are not explicitly stated in \cite{TerDM09}.} thus $\textnormal{rank}([ \, v \otimes I_n \;~ B \, ]) = s < kn$. But then there are $y_1, \ldots y_p \in \mathbb{R}^{kn}$ ($p=kn-s$) with $y_i^T[ \, v \otimes I_n \;~ B \, ]=0$ and therefore $y_i \in \mathcal{N}_\ell(\mathcal{L}(\lambda))$ for all $i=1, \ldots , p$. Thus $\mathcal{L}(\lambda)$ has at least $p$ left minimal indices equal to zero which again contradicts \cite[Thm. 7.2]{TerDM09} for $\mathcal{L}(\lambda)$ seen as an element of $\mathbb{M}_2(P)$.
\end{proof}

In \cite{TerDM09} it is shown that if $P(\lambda)$ is a singular matrix polynomial of degree $k \geq 2$, then none of the pencils
in $\mathbb{DL}(P)$ is a linearization of $P(\lambda).$
Different, larger vector spaces of block-symmetric strong linearizations of matrix polynomials in the monomial basis have been
proposed in \cite{BueDFR15}.

\section{Construction of block-symmetric Pencils}\label{sec5}
This section is dedicated to the construction of pencils in $\mathbb{DM}(P)$ for regular $P(\lambda)$. It turns out that the
characterization
(\ref{thm_charM1}) yields a simple procedure to construct block-symmetric pencils.
As before, assume $P(\lambda)$ to be of the form (\ref{expr_P}) with $\textnormal{deg}(P(\lambda)) \geq 2$. Moreover, let
$\mathcal{L}(\lambda)$ be an element of $\mathbb{M}_1(P)$ as in (\ref{blocksym1}), i.e.
$$\mathcal{L}(\lambda) = [ \, v \otimes I_n \;~ B \, ]F_{\Phi}^P(\lambda) = [ \, v \otimes
\alpha_{k-1}^{-1}P_k \;~ B]
\lambda + [ \, v \otimes I_n \;~ B \, ]F_{\Phi}^P(0). $$
Similar to the block-skew-symmetric
case, $[ \, v \otimes
\alpha_{k-1}^{-1}P_k \;~ B]$ being block-symmetric implies
\[
B= \begin{bmatrix} Z \\ B^\star \end{bmatrix}
\]
with $Z = [ \, v_2 \; v_3 \; \cdots \; v_k \, ] \otimes \alpha_{k-1}^{-1}P_k$ and a $(k-1)n \times (k-1)n$ block-symmetric
matrix $B^\star$. Therefore, considering $B^\star$ as a $(k-1) \times (k-1)$ block matrix with $n \times n$ blocks, it suffices to compute the blocks of the lower triangular part of $B^\star$, that is the blocks $B_{ij}^\star$ with $i \geq j$,\footnote{Notice that, in terms of $B$,
it holds that $B_{s,t} =
B_{t+1,s-1}$ for $t < s$ and $s \geq 2$.} and to only consider $[ \, v \otimes I_n \;~ B \,] F_{\Phi}^P(0)
=\mathcal{L}(0)$ for the remaining derivations.

The block-symmetry certainly requires
\begin{equation}
 (e_i^T \otimes I_n) \mathcal{L}(0) (e_1 \otimes I_n) = (e_1^T \otimes I_n) \mathcal{L}(0) (e_i \otimes I_n).
\label{equ_blocksym}
\end{equation}
As the first block column $v\otimes I_n$ and the first block row $Z$ of $\mathcal{L}(0)$ are already known,
equation (\ref{equ_blocksym}) reads for $2 \leq i \leq k$
\begin{align*} v_i \bigg( - \frac{\beta_{k-1}}{\alpha_{k-1}}P_k &+ P_{k-1} \bigg) - \alpha_{k-2} B_{i,1} \\ & = v_1
P_{k-i} - \frac{ \big( v_{i-1} \gamma_{k-i+1}  + v_{i}
\beta_{k-i} + v_{i+1} \alpha_{k-i-1} \big) }{\alpha_{k-1}} P_k \end{align*}
whereby we set $v_{k+1}= \alpha_{-1}=0$ for $i=k$.\footnote{Terms involving $\alpha_{-1}$ will show up for $i=k$ in
subsequent formulas, too. We will always assume $\alpha_{-1}=0$. } This can easily be solved for the matrix $B_{i,1}$ and yields
\begin{equation} \begin{aligned} B_{i,1} &= \frac{(v_{i-1} \gamma_{k-i+1} + v_{i} (\beta_{k-i} - \beta_{k-1}) + v_{i+1}
\alpha_{k-i-1})P_k}{\alpha_{k-1} \alpha_{k-2}} \\ & \hspace{1cm} + \frac{(v_iP_{k-1} - v_1P_{k-i})}{\alpha_{k-2}}.
\end{aligned} \label{blocksym0} \end{equation}
In this way, the blocks $B_{i,1}$ can be computed for all $i=2, \ldots , k$. Due to the block-symmetry of $B^\star$ this
completely and uniquely determines the first block column and block row of $B^\star$. In the same way, considering
\begin{equation*}
 (e_i^T \otimes I_n) \mathcal{L}(0) (e_2 \otimes I_n) = (e_2^T \otimes I_n) \mathcal{L}(0) (e_i \otimes I_n).
\end{equation*}
gives the equation
\begin{align*} v_i \bigg( P_{k-2} &- \frac{ \gamma_{k-1}}{\alpha_{k-1}}P_k \bigg) - \beta_{k-2}B_{i,1} - \alpha_{k-3} B_{i,2} \\
&= v_2 P_{k-i} - ( \gamma_{k-i+1} B_{2,i-2} + \beta_{k-i} B_{2,i-1} + \alpha_{k-i-1} B_{2,i} ). \end{align*}
It follows from the block-symmetry of $B^\star$ that $B_{2,i-2} = B_{i-1,1}$, $B_{2,i-1} = B_{i,1}$ and
$B_{2,i}=B_{i+1,1}$. Thus we obtain an explicit expression for $B_{i,2}$:
\begin{equation}
\begin{aligned} B_{i,2} &= \frac{( \gamma_{k-i+1}B_{i-1,1} + (\beta_{k-i} - \beta_{k-2})B_{i,1} +
\alpha_{k-i-1}B_{i+1,1})}{\alpha_{k-3}} \\  & \hspace{1cm} + \frac{(v_iP_{k-2} -
v_2P_{k-i})}{\alpha_{k-3}} - v_i \frac{\gamma_{k-1}}{\alpha_{k-3}\alpha_{k-1}}P_k. \end{aligned} \label{blocksym3}
\end{equation}
Due to the block-symmetry it suffices to consider (\ref{blocksym3}) only for $i \geq 3$. Therefore, the blocks $B_{3,2},
\ldots ,
B_{k,2}$ may be computed via (\ref{blocksym3}). Following the same pattern, the equation $(e_i^T \otimes I_n)
\mathcal{L}(0) (e_j \otimes I_n) = (e_j^T \otimes I_n) \mathcal{L}(0) (e_i \otimes I_n)$ yields in its most general form for $j
\geq 3$ and $i \geq j$
\begin{align}
  B_{i,j} &= \frac{\gamma_{k-i+1} B_{i-1,j-1} + ( \beta_{k-i} - \beta_{k-j}) B_{i,j-1} + \alpha_{k-i-1} B_{i+1,j-1} -
\gamma_{k-j+1} B_{i,j-2}}{\alpha_{k-j-1}} \notag \\ & \hspace{1cm} + \frac{(v_iP_{k-j} - v_jP_{k-i})}{\alpha_{k-j-1}}.
  \label{blocksym4}
\end{align}
Hence, we may interpret the blockwise computation of $B$ as some kind of updated
recurrence relation. Moreover, the derivation shows that (\ref{blocksym0}) - (\ref{blocksym4}) are sufficient and necessary for
$\mathcal{L}(\lambda) \in \mathbb{DM}(P)$ being block-symmetric and having ansatz vector $v$. \\

\noindent We summarize the procedure to compute block-symmetric pencils in $\mathbb{M}_1(P)$:
For any regular matrix polynomial $P(\lambda) \in \mathbb{R}^{n \times n}$ expressed in some orthogonal basis as in
(\ref{expr_P}) and of degree $ k \geq 2$ choose any $v \in \mathbb{R}^k$ and compute
$$ B = \begin{bmatrix} Z \\ B^\star \end{bmatrix} \in \mathbb{R}^{kn \times (k-1)n} \qquad B= [B_{i,j}], B_{i,j} \in
\mathbb{R}^{n \times n} $$
according to (\ref{blocksym0}) - (\ref{blocksym4}) and set $Z = [ \, v_2 \; v_3 \; \cdots \; v_k \,] \otimes \alpha_{k-1}^{-1}P_k$.
Then $$\mathcal{L}(\lambda) = [ \, v \otimes I_n \;~ B \, ] F_{\Phi}^P(\lambda)$$ is block-symmetric with ansatz vector $v$.

\begin{example}
  The Chebyshev polynomials of first kind follow the recurrence relation $$ \phi_{j+1}(\lambda) = 2 \lambda
\phi_j(\lambda) - \phi_{j-1}(\lambda) \qquad j \geq 1$$
with $\phi_1(\lambda)= \lambda$ and $\phi_0(\lambda)=1.$ Now let a matrix polynomial $P(\lambda) = P_3
\phi_3(\lambda) + P_2 \phi_2(\lambda) + P_1 \phi_1(\lambda) + P_0 \phi_0(\lambda)$ of degree $3$ be given in the
Chebyshev basis. According to (\ref{stronglin_F}) the strong linearization $F_{\Phi}^P(\lambda)$ has the form
$$ F_{\Phi}^P(\lambda) = \begin{bmatrix} 2 \lambda P_3 + P_2 & P_1 - P_3 & P_0 \\ - \tfrac{1}{2}I_n & \lambda
I_n & - \tfrac{1}{2}I_n \\ 0 & -I_n & \lambda I_n \end{bmatrix}. $$
Using the algorithm for the construction of block-symmetric pencils in $\mathbb{DM}(P)$ we may easily compute
the block-symmetric pencils that correspond to the standard unit vectors $v=e_1, e_2, e_3 \in \mathbb{R}^3$. In
particular we have
\begin{align*} \big[ \, e_1 \otimes I_n \;~ B_1 \, \big] &= \begin{bmatrix} I_n & 0_n & 0_n \\ 0_n &
2(P_3-P_1) & -2P_0 \\ 0_n & -2P_0 & P_3-P_1 \end{bmatrix}, \\ \big[ \, e_2 \otimes I_n \;~ B_2 \, \big] &=
\begin{bmatrix} 0_n & 2P_3 & 0_n \\ I_n & 2P_2 & 2P_3 \\ 0_n & 2P_3 & P_2-P_0 \end{bmatrix}, \\  \big[ \, e_3
\otimes I_n \;~ B_3 \, \big] &= \begin{bmatrix} 0_n & 0_n & 2P_3 \\ 0_n & 4P_3 & 2P_2 \\ I_n & 2P_2 & P_3+P_1
\end{bmatrix}.
\end{align*}
Notice that such pencils need not be (strong) linearizations for $P(\lambda)$. For instance, if $P(\lambda)$
is regular, $\mathcal{L}(\lambda) = [ \, e_3 \otimes I_n \;~ B_3 \, ]F_{\Phi}^P(\lambda)$ can only be a
linearization for $P(\lambda)$ when $P_3$ is nonsingular (due to the anti-lower-block-triangular form of $[ \, e_3 \otimes
I_n \;~ B_3 \, ]$).
\end{example}

Notice that the algorithmic approach for constructing block-symmetric pencils does not require a single
matrix-matrix-multiplication, instead only
scalar-matrix-mul\-ti\-pli\-cations are needed. The complexity of this procedure is $\mathcal{O}(k^2n^2)$, which
also is the complexity
of the construction algorithm presented in \cite[Sec. 7]{NakNT12}. Although there are structural similarities between both
algorithms, they rise from quite different viewpoints.

Fortunately, now we obtain the following corollary without real effort.

\begin{corollary} \label{cor_dimDM}
 For any $n \times n$ regular or singular matrix polynomial $P(\lambda)$ of degree $k \geq 2$ $$\textnormal{dim}(\mathbb{DM}(P)) = k.$$
\end{corollary}
\begin{proof}
First observe that  $\textnormal{dim}(\mathbb{DM}(P)) \geq k$ certainly holds because
$\mathcal{B}_1(\lambda),$ $\ldots,$ $\mathcal{B}_k(\lambda) \in \mathbb{DM}(P)$  with
$\mathcal{B}_j(\lambda) = \big[ \, e_j \otimes I_n \;~ B_j \, \big] F_{\Phi}^P(\lambda)$
are obviously linear independent. Now observe that any
$ \mathcal{L}(\lambda) = \sum_{i=1}^k \alpha_i \mathcal{B}_i(\lambda)$ for arbitrary coefficients $\alpha_i \in
\mathbb{R}$ is block-symmetric with
ansatz vector $v = \sum_{i=1}^k \alpha_i e_i$. Thus, whenever
any $\mathcal{L}^\star(\lambda) \in
\mathbb{DM}(P)$ has ansatz vector $v$, we necessarily have $\mathcal{L}(\lambda) = \mathcal{L}^\star(\lambda)$ due
to the uniqueness of the expressions (\ref{blocksym0}) - (\ref{blocksym4}). Thus
$\textnormal{dim}(\mathbb{DM}(P)) \leq k$ and Corollary \ref{cor_dimDM} follows. \end{proof}

\section{The Eigenvector Exclusion Theorem and the Recovery of left Eigenvectors}
\label{sec:eigenvectorexclusion}
In this section we present a new linearization condition for pencils in $\mathbb{M}_1(P)$ and $\mathbb{M}_2(P)$ that we call
\textit{Eigenvector Exclusion Theorem}. Notice the similarity to Theorem \ref{thm_eetsingular}. 

\begin{theorem}[Eigenvector Exclusion Theorem] \label{thm_eet}
 Let $P(\lambda)$ be an $n \times n$ regular matrix polynomial of degree $k \geq 2$ and assume
$\mathcal{L}(\lambda) \in
\mathbb{M}_1(P)$ with ansatz vector $v \in \mathbb{R}^k$. Then $\mathcal{L}(\lambda)$ is a strong linearization for $P(\lambda)$
if and only if
\begin{equation}
 u^T \big( v \otimes I_n \big) \neq 0 \label{eigenvectorexclusion}
\end{equation}
holds for any left eigenvector $u$ of $\mathcal{L}(\lambda)$ for every eigenvalue $\alpha$ of
$P(\lambda)$.
\end{theorem}

\begin{proof}
We confine ourselves to a sketch of the proof since it is similar to that of Theorem \ref{thm_eetsingular}. Assume that
$\mathcal{L}(\lambda)$ as given in (\ref{thm_charM1}) is a strong linearization
for $P(\lambda)$  and $u \in \mathcal{N}_\ell(\mathcal{L}(\alpha))$ for some eigenvalue
$\alpha$ of $P(\lambda)$, that is $u^T
\mathcal{L}(\alpha)=0$. Then, defining $w \in \mathbb{C}^{kn}$ as
 $$w^T = \big[ \, w_1 \; w_2 \; \cdots \; w_{kn} \, \big] := u^T\big[ \, v \otimes I_n \;~ B \, \big]$$
 and assuming that $u^T( v \otimes I_n) = 0$, we obviously obtain $[ \, w_1 \; \cdots \; w_n \, ] = 0$. Moreover,
since
$u \in \mathcal{N}_\ell(\mathcal{L}(\alpha))$, we have
$$ u^T \mathcal{L}(\alpha) = u^T [ \, v \otimes I_n \;~ B \, ]F_{\Phi}^P(\alpha) = w^T
F_{\Phi}^P(\alpha)=0.  $$
Now a similar argumentation as in the proof of Theorem \ref{thm_eetsingular} gives that $w \equiv 0$, so $u \in \mathcal{N}_\ell([ \, v \otimes I_n \;~ B \, ])$ and $[ \, v \otimes I_n \;~ B \, ]$ is singular. A contradiction since we assumed $\mathcal{L}(\lambda)$ to be a strong linearization for $P(\lambda)$ (see Theorem \ref{thm_stronglin}). On the other hand, whenever $[ \, v \otimes I_n \;~ B \, ]$ is singular, there is a vector $q \in \mathbb{R}^{kn}$ such that $q^T[ \, v \otimes I_n \;~ B \, ]=0$, so $q \in \mathcal{N}_\ell(\mathcal{L}(\alpha))$ for any $\alpha \in \mathbb{C}$.  Now clearly $u^T( v
\otimes I_n) = 0$ holds. The proof follows the same arguments when $\alpha = \infty$ using $\textnormal{rev}_1 \,
\mathcal{L}(0)$ instead of $\mathcal{L}(\alpha)$.

\end{proof}

Now Theorem \ref{thm_eet} enables us to give a statement on the recovery of left eigenvectors for regular matrix polynomials $P(\lambda)$. To this end,
suppose that $\mathcal{L}(\lambda)$ is a strong linearization for $P(\lambda)$, so (\ref{eigenvectorexclusion}) holds for any left eigenvector $u \in \mathbb{R}^{kn}$ for $\mathcal{L}(\lambda)$ for any eigenvalue $\alpha$ of $P(\lambda)$. Then from (\ref{ansatzequation}) we obtain
$$ 0 = u^T \mathcal{L}(\alpha)(\Phi_k(\alpha) \otimes I_n) = u^T(v \otimes I_n)P(\alpha) $$ and therefore, since $u^T(v \otimes I_n) \neq 0$, $u^T(v \otimes I_n)$ is a left eigenvector for $P(\lambda)$ with corresponding eigenvalue $\alpha$. In other words, Theorem \ref{thm_eet} states that for strong linearizations the mapping $u \mapsto (v \otimes I_n)u$ mapping left eigenvectors of $\mathcal{L}(\lambda)$ to left eigenvector of $P(\lambda)$ is injective for any eigenvalue $\alpha$ of $P(\lambda)$. Therefore, when $\mathcal{L}(\lambda)$ is a strong linearization for $P(\lambda)$ we obtain a bijection between the left eigenvectors of $\mathcal{L}(\lambda)$ and the left eigenvectors of $P(\lambda)$.

\begin{corollary} \label{cor:lefteigenvectors}
Let $P(\lambda)$ be an $n \times n$ regular matrix polynomial of degree $k \geq 2$ and $\mathcal{L}(\lambda) \in \mathbb{M}_1(P)$ with ansatz vector $v \in \mathbb{R}^k$ a strong linearization for $P(\lambda)$. Then any left eigenvector $w \in \mathbb{C}^n$ of $P(\lambda)$ with corresponding eigenvalue $\alpha \in \mathbb{C}$ has the form $w = u^T(v \otimes I_n)$ for some left eigenvector $u \in \mathbb{C}^{kn}$ of $\mathcal{L}(\lambda)$ with corresponding eigenvalue $\alpha$.
\end{corollary}

\begin{remark} \label{rem_eigenvectorexclusion}
Certainly, a statement similar to Theorem \ref{thm_eet} holds for pencils in $\mathbb{M}_2(P)$. In particular, whenever $P(\lambda)$ is an  $n
\times n$ regular matrix polynomial of degree $k$, an analogous proof shows that a pencil $\mathcal{L}(\lambda) \in \mathbb{M}_2(P)$ with
ansatz vector $v \in \mathbb{R}^k$ is a strong linearization for $P(\lambda)$ if and only if
$$ (v^T \otimes I_n)u \neq 0 $$
holds for any right eigenvector $u \in \mathcal{N}_r(\mathcal{L}(\alpha))$ for every eigenvalue $\alpha$
of $P(\lambda)$. Obviously this is the same
condition as (\ref{eigenvectorexclusion}) using right instead of left eigenvectors. Of course Corollary \ref{cor:lefteigenvectors} holds in a similar way for $\mathbb{M}_2(P)$ as well.
\end{remark}

We may now prove one direction of the Eigenvalue Exclusion Theorem for generalized ansatz spaces without any
effort. For the monomial basis and general degree-graded bases, this
statement was proven in \cite[Sec. 6]{MacMMM06} and \cite{NakNT12} respectively.

\begin{theorem} \label{thm_eigenvalueet}
 Let $P(\lambda)$ be a regular $n \times n$ matrix polynomial of degree $k$ and assume $\mathcal{L}(\lambda) \in
\mathbb{DM}(P)$ with ansatz vector $v \in \mathbb{R}^k$. Then if $\mathcal{L}(\lambda)$ is a strong linearization
for $P(\lambda)$
no root of the polynomial
$$\Phi_k(\lambda)^Tv = \phi_{k-1}(\lambda)v_k + \phi_{k-2}(\lambda)v_{k-1} + \cdots + \phi_0(\lambda)v_1$$
coincides with an eigenvalue of $P(\lambda)$. Moreover, if $\alpha = \infty$ is an eigenvalue of $P(\lambda),$ then $v_1 \neq 0$.
\end{theorem}

\begin{proof}
Let $\mathcal{L}(\lambda) \in \mathbb{DM}(P)$ be a strong linearization. According to Theorem \ref{thm_eet} we
know that (\ref{eigenvectorexclusion}) holds for any $u \in
\mathcal{N}_\ell(\mathcal{L}(\alpha))$ for every eigenvalue $\alpha$ of $P(\lambda)$.
Moreover, since $\mathcal{L}(\lambda) \in \mathbb{M}_2(P)$, we know from Proposition \ref{prop_eigenvectorrecovery} that any
$u \in \mathcal{N}_\ell(\mathcal{L}(\alpha))$ has the form $u^T = \Phi_k(\alpha)^T \otimes w^T$
for some $w \in
\mathcal{N}_\ell(P(\alpha))$ (or $e_1^T \otimes w^T$ in the case $\alpha = \infty$). Therefore, according
to Theorem \ref{thm_eet}, $\mathcal{L}(\lambda)$ is a linearization for
$P(\lambda)$ if and only if
$$0 \neq \big( \Phi_k(\alpha)^T \otimes w^T \big) \big( v \otimes I_n \big) = \Phi_k(\alpha)^Tv \otimes
w^T. $$
Since $w \neq 0$ this holds if and only if $\Phi_k(\alpha)^Tv \neq 0$. If $\alpha = \infty$ is an eigenvalue of $P(\lambda)$ we obtain according to Theorem \ref{thm_eigenvalueet} $(e_1^T \otimes w^T)(v \otimes I_n) = v_1 \otimes w^T \neq 0$, so $v_1 \neq 0$.
\end{proof}

Notice that the proof of Theorem \ref{thm_eigenvalueet} would have worked using $\mathbb{M}_1(P)$ and $\mathbb{M}_2(P)$ in
reversed roles (see Remark \ref{rem_eigenvectorexclusion}).

\section{A note on other polynomial bases}
\label{sec:polbases}

We would like to emphasize that most of the results in this paper can be proven with just a few ingredients.
In fact, the results of Section 2, 3 and 4 only make use of the ansatz equation and the fact that $F_{\Phi}^P(\lambda)$ was assumed to be a strong linearization for $P(\lambda)$ that satisfies $F_{\Phi}^P(\lambda)( \Phi_k(\lambda) \otimes I_n) = e_1 \otimes P(\lambda)$. In addition, most results in Section 5, 6, 7 and 8 make particularly use of the fact, that the ''anchor pencil`` $F_{\Phi}^P(\lambda)$ has a special upper-block-Hessenberg-structure with $\lambda$ appearing only in the diagonal blocks. In this section we will shortly motivate that these properties are actually all that is necessary for most of the theory developed in this paper.

To this end, let $\Phi = \lbrace \phi_i(\lambda) \rbrace_{i=0}^{\infty}$ be any degree-graded polynomial basis with $\phi_{-1}(\lambda) = 0$ and $\phi_0(\lambda)=1$ that satisfies the recurrence relation
$$ \phi_i(\lambda) = (\lambda - \alpha_i) \phi_{i-1}(\lambda) + \sum_{j=0}^{i-2} \beta_i^j \phi_j(\lambda) \qquad i \geq 1$$
for real coefficients $\alpha_i, i \geq 1$ and $\beta_i^j$ with $i \geq 2$ and $j \leq i-1$. Furthermore let
$$P(\lambda) = \sum_{i=0}^k P_i \phi_i(\lambda)$$
be a matrix polynomial of degree $k \geq 2$.
As in Section \ref{sec3} we define $\Phi_k(\lambda) = \, [\phi_{k-1}(\lambda) \; \phi_{k-2}(\lambda) \; \cdots \; \phi_0(\lambda) \, ]^T \in \mathbb{R}[\lambda]^k$. Now consider the $(k-1) \times k$ matrix pencil
$$ M_{\Phi}(\lambda) = \begin{bmatrix} -1 & (\lambda - \alpha_{k-1}) & \beta_{k-1}^{k-3} & \beta_{k-1}^{k-4} & \cdots & \beta_{k-1}^{1} & \beta_{k-1}^0 \\
 & -1 & (\lambda - \alpha_{k-2}) & \beta_{k-2}^{k-4} & \cdots & \beta_{k-2}^1 & \beta_{k-2}^0 \\
  & & \ddots & \ddots & \ddots & & \vdots \\
  & & & & -1 & (\lambda - \alpha_2) & \beta_2^0 \\ & & & & &-1 & (\lambda - \alpha_1) \end{bmatrix} $$
and note that $M_{\Phi}(\lambda) \Phi_k(\lambda)=0$. In fact it may be proven according to \cite{BueDPSZ} that $M_{\Phi}(\lambda)$ and $\Phi_k(\lambda)^T$ are dual minimal bases. In addition, taking the matrix polynomial $P(\lambda)$ into account, we define the $n \times kn$ matrix pencil
$$ m_{\Phi}^P(\lambda) = \big[ \, (\lambda - \alpha_k)P_k + P_{k-1} \; \; \beta_k^{k-2}P_k + P_{k-2}  \; \, \cdots \; \, \beta_k^0 P_k+P_0 \, \big]. $$
As before $m_{\Phi}^P(\lambda)(\Phi_k(\lambda) \times I_n) = P(\lambda)$ holds. According to \cite{BueDPSZ}
\begin{equation} G_{\Phi}^P(\lambda) = \begin{bmatrix} m_{\Phi}^P(\lambda) \\ M_{\Phi}(\lambda) \otimes I_n \end{bmatrix} \in \mathbb{R}_1[\lambda]^{kn \times kn} \label{general_stronglin} \end{equation}
is a strong block minimal bases pencil for $P(\lambda)$ and therefore a strong linearization for $P(\lambda)$. Defining the vector space $\mathbb{M}_1(P)$ as the set of $kn \times kn$ matrix pencils that satisfy (\ref{ansatzequation}) for some ansatz vector $v \in \mathbb{R}^k$ we obtain a similar result to Theorem \ref{thm_master1}, i.e. every pencil $\mathcal{L}(\lambda)$ that satisfies (\ref{ansatzequation}) may expressed as
$$ \mathcal{L}(\lambda) = \big[ \, v \otimes I_n \;~ B \, \big] G_{\Phi}^P(\lambda)$$
for some matrix $B \in \mathbb{R}^{kn \times (k-1)n}.$ The proof works essentially as for Theorem \ref{thm_master1}.
Since $G_{\Phi}^P(\lambda)$ is a strong linearization for $P(\lambda)$, the mapping
$$\chi: \big[ \, v \otimes I_n \;~ B \, \big] \mapsto \big[ \, v \otimes I_n \;~ B \, \big] G_{\Phi}^P(\lambda)$$
will be injective, so we obtain Corollary \ref{cor_dimension}. Moreover, it is not hard to see that Corollary \ref{cor_lincondition1}, Theorem \ref{thm_stronglin} and Corollary \ref{cor_generic} still hold. For the eigenvector recovery result from Proposition \ref{prop_eigrec1} and the linearization condition from Proposition \ref{prop_eigrec2} once more nothing but the special form of $\mathcal{L}(\lambda)$ is required, so these results still hold for any other degree-graded polynomial basis.

From Section \ref{sec:singular} on the proofs presented in this paper make particularly use of the upper-block-Hessenberg-structure of $F_{\Phi}^P(\lambda)$.
Since this structure is inherited by $G_{\Phi}^P(\lambda)$ the results from Theorem \ref{thm_eetsingular}, Corollary \ref{singular_lin} and Theorem \ref{thm_stronglinsing} will still hold. Moreover, also the ideas behind the construction procedure for block-symmetric pencils presented in Section \ref{sec5} will work. We illustrate this with an example.

\begin{example}
Let $\phi_0(\lambda)=1$ and $\phi_i(\lambda) = \lambda \phi_{i-1}(\lambda) + 1$ for $i \geq 1$. This defines a nonstandard degree-graded polynomial basis consisting of the polynomials $\phi_0(\lambda)=1, \phi_1(\lambda)=\lambda+1, \phi_2(\lambda) = \lambda^2 + \lambda + 1$ and so on. Now consider an $n \times n$ matrix polynomial
$$P(\lambda) = P_4 \phi_4(\lambda) + P_3 \phi_3(\lambda) + P_2 \phi_2(\lambda) + P_1 \phi_1(\lambda) + P_0 \phi_0(\lambda)$$
of degree $k=4$. According to (\ref{general_stronglin}) $G_{\Phi}^P(\lambda)$ is given as
$$ G_{\Phi}^P(\lambda) = \begin{bmatrix} \lambda P_4 + P_3 & P_2 & P_1 & P_0 + P_4 \\ -I_n & \lambda I_n & 0 & I_n \\ 0 & -I_n & \lambda I_n & I_n \\ 0 & 0 & - I_n & (\lambda + 1)I_n \end{bmatrix}$$
which is always a strong linearization for $P(\lambda)$.\footnote{Notice that $G_{\Phi}^P(\lambda)$ does not have the tridiagonal block-structure in the lower $(k-1)n \times kn$ block as $F_{\Phi}^P(\lambda)$ due to the identity block in the position $(2,4)$.} Now, we may adapt the construction procedure for block-symmetric pencils to this situation looking for a matrix pencil
$$ \mathcal{L}(\lambda) = \begin{bmatrix} v_1 I_n & v_2 P_4 & v_3 P_4 & v_4 P_4 \\ v_2 I_n & B_{21} & B_{22} & B_{23} \\ v_3 I_n & B_{31} & B_{32} & B_{33} \\ v_4 I_n & B_{41} & B_{42} & B_{43} \end{bmatrix} G_{\Phi}^P(\lambda)$$
with $B_{22} = B_{31}$, $B_{23}=B_{41}$ and $B_{33} = B_{42}$ which is block-symmetric. The construction procedure gives
\begin{align*}
B_{21} &= -v_1P_2 + v_3P_4 v_2P_3 \\ B_{31} &= -v_1P_1 + v_4P_4 + v_3P_3 \\
B_{41} &= -v_1(P_0+P_4) - v_2P_4 - v_3P_4 -v_4P_4 + v_4P_3 \\
B_{32} &= - v_2P_1 + v_3 P_2 + B_{41} \\
B_{42} &= -v_2(P_0+P_4) - B_{21} - B_{31} - B_{41} + v_4P_2 \\
B_{43} &= -v_3(P_0+P_4) - B_{31} - B_{32} - B_{42} + v_4P_1
\end{align*}
A straightforward computation shows that $\mathcal{L}(\lambda)$ in fact becomes block-sym\-metric.
\end{example}
Finally, Theorem \ref{thm_eigenvalueet} will also hold in the context of degree-graded polynomial bases since as before nothing but the special form of $\mathcal{L}(\lambda)$ and the upper-block-Hessenberg-structure of $G_{\Phi}^P(\lambda)$ are required for the proof.

\section{Conclusion}\label{sec6}

We presented a rigorous generalization of the results obtained in \cite{MacMMM06} to orthogonal polynomial bases. Although the
extension of the concepts from \cite{MacMMM06} to nonstandard bases has already been considered, it was one of our main aims to
present the subject in a cohered and concise manner introducing some new aspects without drawing on deeper theoretical
results.
Setting up the generalized
ansatz spaces as introduced in \cite[Sec. 4.2]{MacMMM06}, we were able to characterize the elements in these spaces nicely,
obtain simple linearization conditions and prove statements on the space dimension or the genericity of
linearizations without
any effort. Moreover, we gave a condition equivalent to the full $Z$-rank condition for singular matrix polynomials. A basic and short algebraic
proof on the fact that double generalized ansatz spaces contain
entirely block-symmetric pencils using a rather surprising argument on block-skew-symmetric pencils is presented.
We
derived an intuitive procedure to construct block-symmetric pencils in generalized ansatz spaces and presented the
Eigenvector Exclusion Theorem, which is an analog of the eigenvalue exclusion theorem for non-block-symmetric
pencils. Furthermore, the proofs of the results in this paper need just a few ingredients which may make it easy
to extend them to similar results for degree-graded polynomial bases.

\section*{Acknowledgements}
Both authors would like to thank Javier  P\'erez for his valuable comments on an earlier version of this
manuscript.
The second author would like to thank D. Steven Mackey and Froil\'{a}n M. Dopico for their friendly remarks on this work at the
2016 ILAS conference in Leuven, Belgium.

\end{document}